\documentclass[]{interact}
\usepackage{latexsym}
\usepackage{anyfontsize} 
\usepackage{tikz} 
\usetikzlibrary{cd} 
\usepackage{color}
\usepackage{hyperref}
\usepackage{url}
\usepackage{breakurl}
\newcommand{\bburl}[1]{\textcolor{blue}{\url{#1}}}

\makeatletter
\newcommand{\monthyear}[1]{%
  \def\@monthyear{\uppercase{#1}}}
\newcommand{\volnumber}[1]{%
  \def\@volnumber{\uppercase{#1}}}
\makeatother

\theoremstyle{plain}
\numberwithin{equation}{section} 
\newtheorem{thm}{Theorem}[section]
\newtheorem{theorem}[thm]{Theorem}

\numberwithin{table}{section} 
\numberwithin{figure}{section}

\begin{document}

\monthyear{Month Year}
\volnumber{Volume, Number}
\setcounter{page}{1}

\title{Water Cells in Compositions of 1s and 2s}

\author{
\name{Brian Hopkins\textsuperscript{a}\thanks{CONTACT Brian Hopkins email: bhopkins@saintpeters.edu} and Aram Tangboonduangjit\textsuperscript{b}}
\affil{\textsuperscript{a}Department of Mathematics and Statistics,
                Saint Peter's University,
               Jersey City, New Jersey,
	07306, USA; \textsuperscript{b} Mahidol University International College,
               Mahidol University,
               Nakhon Pathom,
               73170, Thailand}
}

\maketitle

\begin{abstract}
Mansour and Shattuck introduced the notion of water cells for integer compositions in 2018.  We focus on compositions with parts restricted to 1 and 2 and consider the array of counts for such compositions of $n$ with $k$ water cells, establishing generating functions for the columns and diagonal sums, recurrences within the array in the spirit of Pascal's lemma, and connections to other restricted compositions.  Most of our proofs are combinatorial, but we also make use of Riordan arrays.\end{abstract}

\begin{keywords}
	Integer compositions; water cells; recurrence relations; combinatorial proofs; Riordan arrays
\end{keywords}

\section{Introduction and background}
Given an integer $n \ge 1$, a composition of $n$ is an ordered collection $(c_1, \ldots, c_t)$ of positive integers such that $\sum_i c_i = n$.  We write $C(n)$ for the set of compositions of $n$ and $c(n) = |C(n)|$ for their count.  For example, 
\[ C(4) = \{(4), (3,1), (2,2), (2,1,1), (1,3), (1,2,1), (1,1,2), (1,1,1,1)\}\]
and $c(4) = 8$.  We sometimes use superscripts to denote repetition, e.g., writing $(1^4)$ rather than $(1,1,1,1)$.   
By convention, $c(0) = 1$ counting the empty composition.

MacMahon proved $c(n) = 2^{n-1}$ for $n \ge 1$ by a combinatorial argument \cite{m15} equivalent to the following.  Represent a composition of $n$ as a tiling of a $1 \times n$ board where a part $k$ corresponds to a $1 \times k$ rectangle.  The composition is determined by $n-1$ binary choices we call cut and join: A cut \texttt{C} at a juncture separates two parts while a join \texttt{J} is internal to a part with $k \ge 2$.  Figure \ref{fig0} shows this representation of the composition $(3,1)$ which has cut-join sequence \texttt{JJC}.  Further, MacMahon defined the conjugate of a composition obtained by switching the cuts and joins; Figure \ref{fig0} also shows the tiling representation of $(1,1,2)$ with cut-join sequence \texttt{CCJ}, the conjugate of $(3,1)$. 

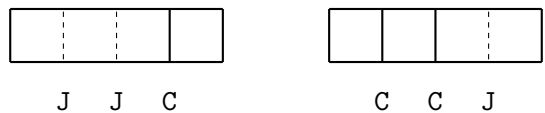
\begin{figure}[ht]
\begin{center}
\begin{picture}(200,40)
\setlength{\unitlength}{2pt}
\multiput(10,10.5)(0,2){5}{\line(0,1){1}}
\multiput(20,10.5)(0,2){5}{\line(0,1){1}}

\multiput(90,10.5)(0,2){5}{\line(0,1){1}}

\thicklines
\put(0,10){\line(1,0){40}}
\put(0,20){\line(1,0){40}}
\put(0,10){\line(0,1){10}}
\put(30,10){\line(0,1){10}}
\put(40,10){\line(0,1){10}}

\put(60,10){\line(1,0){40}}
\put(60,20){\line(1,0){40}}
\put(60,10){\line(0,1){10}}
\put(70,10){\line(0,1){10}}
\put(80,10){\line(0,1){10}}
\put(100,10){\line(0,1){10}}

\put(8.5,1){\texttt{J}}
\put(18.5,1){\texttt{J}}
\put(28.5,1){\texttt{C}}

\put(68.5,1){\texttt{C}}
\put(78.5,1){\texttt{C}}
\put(88.5,1){\texttt{J}}

\end{picture}
\caption{The tiling representations of compositions $(3,1)$ with cut-join sequence \texttt{JJC} and its conjugate $(1,1,2)$ with cut-join sequence \texttt{CCJ}.} \label{fig0}
\end{center}
\end{figure}

We will also reference partitions of $n$ which, in contrast to compositions, are unordered collections of parts.  For example,
\[P(4) = \{(4), (3,1), (2,2), (2,1,1), (1,1,1,1)\}\]
where we follow the convention of listing parts in nonincreasing order.
A generalized notion of partitions allows for parts to appear in multiple ``colors.''  These colored parts are also unordered.  For example, the partitions of 4 where there are two colors of 2 available, denoted by subscripts, are
\[\{(4), (3,1), (2_1, 2_1), (2_1, 2_2), (2_2, 2_2), (2_1, 1, 1), (2_2, 1, 1), (1,1,1,1)\};\]
notice that $(2_2, 2_1)$ is not listed as it is equivalent to the partition $(2_1, 2_2)$.

Let $C_{12}(n)$ denote the compositions of $n$ whose parts are all in the set $\{1,2\}$.  From above, we have $c_{12}(4) = |C_{12}(4)| = 5$.  It has been known since ancient India \cite{s85} that $c_{12}(n) = F_{n+1}$ where $F_n$ denotes the Fibonacci numbers defined as $F_0 = 0$, $F_1 = 1$, and $F_n = F_{n-1} + F_{n-2}$ for $n \ge 2$.

In 2018, Mansour and Shattuck introduced the notion of water cells of compositions based on a different graphical representation \cite{ms18}.  The bargraph representation of a composition $(c_1, \ldots, c_t)$ consists of $t$ columns starting from a horizontal line where column $i$ consists of $c_i$ square cells.  A water cell is a square outside of the bargraph representation that would ``hold water'' poured over the shape.  Figure \ref{fig1} shows the bargraph for the composition $(1,2,1,4,2,4,1,2,1,3) \in C(21)$ which has 8 water cells.  See also Blecher, Brennan, and Knopfmacher \cite{bbk18}, who seem to have come upon the same notion independently.  (They would say the composition has capacity 8.)

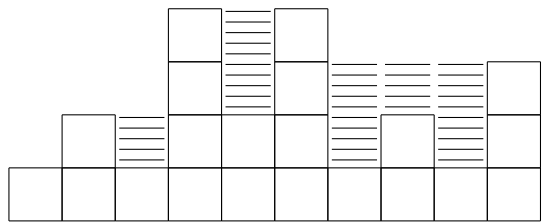
\begin{figure}[ht]
\begin{center}
\begin{picture}(210,80)
\setlength{\unitlength}{2pt}
\def\ccc{\put(0,0){\line(1,0){10}}\put(10,0){\line(0,1){10}}
\put(10,10){\line(-1,0){10}}\put(0,10){\line(0,-1){10}}}
\def\bbb{\multiput(0.75,1.5)(0,2){5}{\line(1,0){8.5}}}
\put(0,0){
\multiput(0,0)(0,10){1}{\ccc}
\multiput(10,0)(0,10){2}{\ccc}
\multiput(20,0)(0,10){1}{\ccc}
\multiput(30,0)(0,10){4}{\ccc}
\multiput(40,0)(0,10){2}{\ccc}
\multiput(50,0)(0,10){4}{\ccc}
\multiput(60,0)(0,10){1}{\ccc}
\multiput(70,0)(0,10){2}{\ccc}
\multiput(80,0)(0,10){1}{\ccc}
\multiput(90,0)(0,10){3}{\ccc}
\multiput(20,10)(0,10){1}{\bbb}
\multiput(40,20)(0,10){2}{\bbb}
\multiput(60,10)(0,10){2}{\bbb}
\multiput(70,20)(0,10){1}{\bbb}
\multiput(80,10)(0,10){2}{\bbb}
}
\end{picture}
\caption{The bargraph of $(1,2,1,4,2,4,1,2,1,3)$ and its water cells.} \label{fig1}
\end{center}
\end{figure}

In this work, we restrict our attention to the compositions $C_{12}(n)$.  Let $W(n,k)$ be the compositions in $C_{12}(n)$ with exactly $k$ water cells and $w(n,k) = |W(n,k)|$.  (Because we are only considering compositions using parts 1 and 2, the more precise notation $W_{12}(n,k)$ is unnecessary here.)  Figure \ref{fig2} provides small examples of these compositions:  $(2,1,2) \in C_{12}(5)$ is the unique smallest composition with one water cell while $(2,1,1,2) \in C_{12}(6)$ and $(2,1,2,1,2) \in C_{12}(8)$ are among the compositions with two water cells.

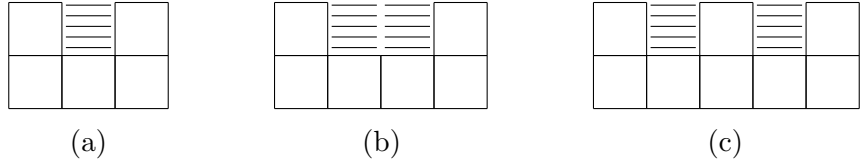
\begin{figure}[ht]
\begin{center}
\begin{picture}(320,70)
\setlength{\unitlength}{2pt}
\def\ccc{\put(0,0){\line(1,0){10}}\put(10,0){\line(0,1){10}}
\put(10,10){\line(-1,0){10}}\put(0,10){\line(0,-1){10}}}
\def\bbb{\multiput(0.75,1.5)(0,2){5}{\line(1,0){8.5}}}
\put(0,10){
\multiput(0,0)(0,10){2}{\ccc}
\multiput(10,0)(0,10){1}{\ccc}
\multiput(20,0)(0,10){2}{\ccc}
\multiput(50,0)(0,10){2}{\ccc}
\multiput(60,0)(0,10){1}{\ccc}
\multiput(70,0)(0,10){1}{\ccc}
\multiput(80,0)(0,10){2}{\ccc}
\multiput(110,0)(0,10){2}{\ccc}
\multiput(120,0)(0,10){1}{\ccc}
\multiput(130,0)(0,10){2}{\ccc}
\multiput(140,0)(0,10){1}{\ccc}
\multiput(150,0)(0,10){2}{\ccc}
\multiput(10,10)(0,10){1}{\bbb}
\multiput(60,10)(0,10){1}{\bbb}
\multiput(70,10)(0,10){1}{\bbb}
\multiput(120,10)(0,10){1}{\bbb}
\multiput(140,10)(0,10){1}{\bbb}
\put(11.5,-8){(a)}
\put(66.5,-8){(b)}
\put(131.5,-8){(c)}
}
\end{picture}
\caption{The bargraphs of (a) $(2,1,2)$ with one water cell, (b) $(2,1,1,2)$ and (c) $(2,1,2,1,2)$ each with two water cells.} \label{fig2}
\end{center}
\end{figure}

Narrowing a very general result of Mansour and Shattuck \cite[Theorem 2]{ms18} to our context gives the generating function
\begin{equation}
\sum w(n,k) q^n z^k = \frac{1}{1-q} + \frac{q^2(1-zq)}{(1-q)^2(1-zq-q^2)} \label{wcgf}
\end{equation}
(see also Blecher, Brennan, and Knopfmacher \cite[p.\ 4]{bbk18}).  Note that with $z=1$ this reduces to $1/(1-q-q^2)$, a generating function for $F_{n+1}$.

In Section 2, we consider the array of numbers $w(n,k)$ and establish recurrences in the resulting triangle of numbers and also for its diagonal sums.  We also give bijections to other classes of restricted compositions as well as certain partitions with colored parts.  

Although most of our arguments are combinatorial, in Section 3 we explain connections to Riordan arrays and use them to complete one proof.  These structures, defined in 1991 by Shapiro and collaborators \cite{sgww91}, succinctly describe certain number triangles and allow for simple determination of row sums, alternating row sums, diagonal sums, etc.  Those unfamiliar with the Riordan group (there is a group structure for these arrays) are invited to read a recent survey article \cite{dfs24} by Shapiro et al.

\section{Water cells in $C_{12}(n)$}

Table \ref{wtab} provides the $w(n,k)$ values for small $n$ and $k$, the number of compositions in $C_{12}(n)$ with exactly $k$ water cells.  As mentioned for Figure \ref{fig2}, the smallest composition with one water cell is $(2,1,2) \in C_{12}(5)$, thus the $k = 1$ column has its first nonzero value in row $n = 5$.  Because we are partitioning the compositions of $C_{12}(n)$ by the number of water cells, the row sums are Fibonacci numbers.

\begin{table}[ht]
\caption{The $w(n,k)$ values for $0 \le n \le 14$ and $0 \le k \le 10$.} 
\medskip
\begin{center}
\renewcommand{\arraystretch}{1.3}
\begin{tabular}{r|rrrrrrrrrrr}
$n \backslash k$ & 0 & 1 & 2 & 3 & 4 & 5 & 6 & 7 & 8 & 9 & 10 \\ \hline
0 & 1 \\
1 & 1 \\
2 & 2 \\
3 & 3 \\
4 & 5 \\
5 & 7   &   1 \\
6 & 10  &  2    & 1 \\
7 & 13  &  5    & 2    & 1 \\
8 & 17  &   8   &  6   & 2   & 1 \\
9 & 21  & 14   & 10  & 7   & 2   & 1 \\
10 & 26  & 20   & 20  & 12 & 8   & 2 &  1 \\
11 & 31  & 30   & 30  & 27 & 14  & 9 &  2 &  1  \\
12 & 37  & 40   & 50  & 42 &  35 & 16 & 10 & 2  & 1 & \phantom{10} \\
13 & 43  & 55   & 70  & 77 &  56 & 44 & 18 & 11 & 2 &   1 \\
14 & 50  & 70   & 105 & 112 & 112 &  72 & 54 & 20 & 12 & 2 & 1
\end{tabular}
\end{center}
\label{wtab}
\end{table}

Below, we establish a generating function for each column of the $w(n,k)$ irregular triangular array.  (The first six columns are A033638, A006918, A096338, A177747, A299337, and A178440 in \cite{o}, respectively.)  Next, we show several additional relations and a direct formula for the first column, then show how other columns have more efficient recurrences when values from other columns are incorporated.  In the last result of this section, we show that the diagonal sums of $w(n,k)$ count compositions where only the first and last parts are allowed to be odd.

\subsection{Water cell columns}
We begin with the first column which, as suggested by the irregular shape of the number array in Table \ref{wtab}, requires individual attention.  Note that a composition without a water cell is weakly unimodal, i.e., $c_1 \le \cdots \le c_k \ge \cdots \ge c_t$ for some $1 \le k \le t$.

\begin{theorem} \label{wum}
The generating function for $w(n,0)$ is
\begin{equation}
 \sum_{n \ge 0} w(n,0) q^n = \frac{1}{1-q} + \frac{q^2}{(1-q)^2(1-q^2)} = \frac{1-q+q^3}{1 - 2 q + 2 q^3 - q^4}. \label{w0gf}
 \end{equation}

Also, for each $n \ge 0$, the following values are equal.
\begin{enumerate}
\renewcommand{\theenumi}{\alph{enumi}}
\item $w(n,0)$,
\item $\lfloor n^2/4 \rfloor + 1$,
\item $w(n-2,0) + n - 1$ (for $n \ge 2$),
\item $2w(n-1,0) - 2w(n-3,0) + w(n-4,0)$ (for $n \ge 4$).
\end{enumerate}
\end{theorem}

\begin{proof}
For the generating function, rather than work from \eqref{wcgf}, we give a combinatorial argument.  A composition in $C_{12}(n)$ with no water cells is weakly unimodal, so it has the form $(1^a, 2^b, 1^c)$ for nonnegative $a,b,c$.  One possibility is the composition $(1^n) \in C_{12}(n)$ accounted for in the generating function by the summand $1/(1-q)$.  The second summand, $q^2/((1-q)^2(1-q^2))$, counts partitions of $n$ with parts 1 and 2 where there is at least one part 2 and there are two colors of parts 1.  These correspond to compositions in $C_{12}(n)$ with no water cells and at least one part 2: any partition parts $1_1$ correspond to the initial run $1^a$ in the composition, additional parts 2 from the partition contribute to $2^b$, and any partition parts $1_2$ correspond to the terminal run $1^c$.  (For example, the partition $(2,2,1_1,1_1,1_2)$ corresponds to the composition $(1,1,2,2,1)$.)  The second generating function expression follows from algebraic manipulation.

For the direct formula (b), consider the number of parts 2, which must be adjacent.  Suppose first that $n = 2m$.  There is one composition counted by $w(n,0)$ that has no parts 2, the composition $(1^n)$.  There are $2m-1$ compositions with a single part 2, as it can be before or after any of the $2m-2$ parts 1.  There are $n-3$ compositions with the block $(2,2)$ before or after any of the $n-4$ parts 1.  This continues to the single composition $(2^m)$, giving
\[w(2m,0) = 1 + (2m-1) + (2m-3) + \cdots + 1 = m^2 + 1.\]
The analogous sum for the case $n = 2m+1$ is
\[w(2m+1,0) = 1 + 2m + (2m-2) + \cdots + 2 = m^2 + m + 1\]
and, in both cases, these match $\lfloor n^2/4 \rfloor + 1$.

For the nonhomogeneous recurrence (c), we describe how to construct $W(n,0)$ from most of $W(n-2,0)$ and some other compositions.  We actually show 
\[w(n,0) = (w(n-2,0) - 1) + (n-1) + 1.\]
For each of the $w(n-2,0) - 1$ compositions of $W(n-2,0)$ that contain at least one part 2 in a single run, add another part 2 to that run.  There are $n-1$ compositions in $W(n,0)$ with a single part 2 in every possible positions before and after the $n-2$ parts 1.  Finally, $(1^n) \in W(n,0)$.  In the reverse direction, from the compositions in $W(n,0)$ with two or more parts 2, remove one of them to make a composition in $W(n-2,0)$; this leaves $n-1$ compositions in $W(n,0)$ with a single part 2 and also $(1^n)$.

The linear recurrence (d) follows from the generating function \eqref{w0gf}, but we give a combinatorial proof, establishing the bijection
\[W(n,0) \cup 2W(n-3,0) \cong 2W(n-1,0) \cup W(n-4,0)\]
where the coefficient 2 means two copies of a set.  The explanation of the bijection is simplified by defining the increasable subset $W^i(n,0)$ of $W(n,0)$ as $(1^n)$ and the compositions whose last two parts are $(2,1)$, i.e., the compositions where the last part can be increased by one without creating a water cell.  Let $W^j(n,0) = W(n,0) \setminus W^i(n,0)$.  We show 
\begin{gather*}
W(n,0) \cong W(n-1,0) \cup W^i(n-1,0), \\
2W(n-3,0) \cong\ W^j(n-1,0) \cup W(n-4,0)
\end{gather*}
from which the claim follows.

For the first bijection, decrease the last part of each composition in $W(n,0)$ by one.  Those with last part 1 are mapped bijectively to $W(n-1,0)$.  The compositions in $W(n,0)$ with last part 2 are mapped into another $W(n-1,0)$ set, namely the compositions ending in $(2,1)$ and $(1^{n-1})$, i.e., $W^i(n-1,0)$.

For the reverse map of the first bijection, add a part 1 at the end of the compositions in $W(n-1,0)$ to obtain the compositions of $W(n,0)$ with last part 1.  Increase the last part of the compositions in $W^i(n-1,0)$ by one (allowed by definition), giving the compositions of $W(n,0)$ with last part 2.

For the second bijection, in the first set $W(n-3,0)$, for the compositions that end in 1, removing that last part establishes a bijection to $W(n-4,0)$.  For the compositions in the first set $W(n-3,0)$ that end in 2, add another part 2.  In the second set $W(n-3,0)$, send $(1^{n-3})$ to $(1^{n-3},2)$ and add the parts $(1,1)$ at the end of the other compositions.  Together these comprise the set $W^j(n-1,0)$: those ending with a run of two or more parts 2, the composition ending in a single 2, and those with at least one part 2 ending in $(1,1)$, exactly the compositions where increasing the last part yields a composition outside $C_{12}(n-1)$ or a composition in $C_{12}(n-1)$ with a water cell.

For the reverse map of the second bijection, for the compositions of $W^j(n-1,0)$ that end in $(1,1)$, removing those last two parts establishes a bijection to $W(n-3,0)$ except for $(1^{n-3})$ (since $(1^{n-1}) \in W^i(n-1,0)$).  Since no composition of $W^j(n-1,0)$ ends in $(2,1)$ (that would be increasable), the remaining compositions have last part 2, which we remove.  One of those is $(1^{n-3},2) \in W^j(n-1,0)$ whose image $(1^{n-3})$ completes the first $W(n-3,0)$.  The other compositions in $W^j(n-1,0)$ that end in 2 have additional parts 2: These must be in a run of parts 2 at the end, else there would be a water cell.  Thus the images in the second $W(n-3,0)$ comprise all the compositions in that $W(n-3,0)$ that end in 2.  Finally, add a part 1 at the end of each composition in $W(n-4,0)$ to complete the second $W(n-3,0)$ with the compositions that end in 1.
\end{proof}

See Table \ref{wumd} for an example of the bijections in the last part of the proof.

\begin{table}[ht]
\caption{The bijections of the proof of Theorem \ref{wum} (d) for $n= 6$.} 
\medskip
\centering
\begin{tabular}{rcl}
$W(6,0)$ & $\longleftrightarrow$ & $W(5,0) \cup W^i(5,0)$ \\ \hline
$(2,2,1,1)$ & & $(2,2,1)$ \\
$(2,1,1,1,1)$ & & $(2,1,1,1)$ \\
$(1,2,2,1)$ & & $(1,2,2)$ \\
$(1,2,1,1,1)$ & & $(1,2,1,1)$ \\ 
$(1,1,2,1,1)$ & & $(1,1,2,1)$ \\
$(1,1,1,2,1)$ & & $(1,1,1,2)$ \\
$(1,1,1,1,1,1)$ & & $(1,1,1,1,1)$ \\ \cline{3-3}
$(2,2,2)$ & & $(2,2,1)$  \\
$(1,1,2,2)$ & & $(1,1,2,1)$ \\
$(1,1,1,1,2)$ & & $(1,1,1,1,1)$ 
\end{tabular}
\begin{tabular}{rcl}
$2W(3,0)$ & $\longleftrightarrow$ & $W(2,0) \cup W^j(5,0)$ \\ \hline
$(2,1)$ & & $(2)$ \\
$(1,1,1)$ & & $(1,1)$ \\ \cline{3-3}
$(1,2)$ & & $(1,2,2)$ \\ \cline{1-1}
$(1,1,1)$ & & $(1,1,1,2)$ \\
$(2,1)$ & & $(2,1,1,1)$ \\ 
$(1,2)$ & & $(1,2,1,1)$ 
\end{tabular}
\label{wumd}
\end{table}

It follows from the proof of Theorem \ref{wum} (d) that the number of increasable compositions in $W(n,0)$ satisfies $\vert W^i(n,0) \vert = \lceil n/2 \rceil$ (the integer ceiling function).  The complementary sequence $\vert W^j(n,0) \vert$ matches A004652 in \cite{o}.

Next we determine the generating function for each sequence $w(n,k)$ with $k \ge 1$ by extending the connection between the appropriate elements of $C_{12}(n)$ and partitions with parts restricted to 1 and 2 which may appear in multiple colors.

\begin{theorem} \label{wcolgf}
For a given $k \ge 1$, the generating function for the sequence $w(n,k)$ is
\[ \sum w(n,k) q^n = \frac{q^{k+4}}{(1-q)^2(1-q^2)^{k+1}}.\]
\end{theorem}

The proof builds on the colored partition argument in the verification of \eqref{w0gf}. 

\begin{proof}
Again, we give a combinatorial proof.

As suggested in Figure \ref{fig2} (a) and (b), the smallest composition with $k$ water cells is $(2,1^k, 2) \in C_{12}(k+4)$.  We describe how to construct, for all $n$, all compositions in $C_{12}(n)$ with exactly $k$ water cells from $(2,1^k, 2)$.

The $q^{k+4}$ term in the numerator of the generating function corresponds to $(2,1^k, 2)$. The denominator of the generating function accounts for partitions made of two colors of parts 1 and $k+1$ colors of parts 2.  The parts of such a partition of $n-k-4$ are incorporated into the composition $(2,1^k, 2)$ to make a composition in $C_{12}(n)$ as follows.  

Any partition parts $1_1$ correspond to an initial run of parts 1 in the composition, and any partition parts $1_2$ correspond to a terminal run of parts 1.  Note that because these new composition parts 1 are not bounded by parts 2 on both sides, they do not contribute any water cells.  

Any partition parts $2_1$ correspond to additional parts 2 in the composition placed after the first part 2 in $(2,1^k, 2)$, and any partition parts $2_{k+1}$ correspond to additional parts 2 in the composition placed after the last part 2 in $(2,1^k, 2)$.  For each $2 \le i \le k$, any partition parts $2_i$ correspond to parts 2 in the composition placed between the $(i-1)$st part 1 and the $i$th part 1 in $(2,1^k, 2)$.  None of these parts 2 impact the number of water cells, so the resulting composition in $C_{12}(n)$ has exactly $k$ water cells.
\end{proof}

Note that the statement about the smallest composition with $k$ water cells in the proof is true only for our setting of $C_{12}(n)$.  Among all compositions, the smallest with six water cells is $(3,1,1,1,3) \in C(9)$, not $(2,1^6,2) \in C(10)$.

It follows from Theorem \ref{wcolgf} that, for a fixed $k \ge 1$, the sequence $w(n,k)$ satisfies a degree $2k+4$ linear recurrence.  We show, however, that using values from other columns allows for much simpler recurrences.  In the next result, for example, we establish that $w(n,1)$ depends on just the two terms $w(n-2,1)$ and $w(n-3,0)$ rather than involving $w(n-6,1)$.

\begin{theorem} \label{wc1}
For $n \ge 5$, \[w(n,1) = w(n-2,1) + w(n-3,0) - 1.\]
\end{theorem}

\begin{proof}
We establish a bijection
\[W(n,1) \cong W(n-2,1) \cup \{W(n-3,0) \setminus (1^{n-3})\}\]
from which the claim follows.

Given a composition in $W(n,1)$, it has either a single part 2 after its water cell or a run of at least two parts 2 after its water cell.  If there is a single part 2 after the water cell, then remove the part 1 beneath the water cell and the next part 2, leaving a composition in $W(n-3,0)$ with at least one part 2 (preceding the removed water cell).  If the composition in $W(n,1)$ has a run of two or more parts 2 after the water cell, remove the first of those parts 2, leaving a composition in $W(n-2,1)$.    

For the reverse map, a composition in $W(n-2,1)$ must contain consecutive parts $(2,1,2)$ to have a water cell; add another part 2 after the water cell to make a composition in $W(n,1)$.  Given a composition in $W(n-3,0)$ with at least one part 2, add the parts $(1,2)$ after the last part 2 to make a composition in $W(n,1)$.
\end{proof}

See Table \ref{wc1ex} for an example of the bijection.

\begin{table}[ht]
\caption{The bijections of the proof of Theorem \ref{wc1} for $n= 8$.} 
\medskip
\centering
\begin{tabular}{rcl}
$W(8,1)$ & $\longleftrightarrow$ & $\{W(5,0) \setminus (1^5)\} \cup W(6,1)$ \\ \hline
$(2,2,1,2,1)$ & & $(2,2,1)$ \\
$(2,1,2,1,1,1)$ & & $(2,1,1,1)$ \\ 
$(1,2,2,1,2)$ & & $(1,2,2)$ \\
$(1,2,1,2,1,1)$ & & $(1,2,1,1)$ \\
$(1,1,2,1,2,1)$ & & $(1,1,2,1)$ \\ 
$(1,1,1,2,1,2)$ & & $(1,1,1,2)$  \\ \cline{3-3}
$(2,1,2,2,1)$ & & $(2,1,2,1)$ \\
$(1,2,1,2,2)$ & & $(1,2,1,2)$
\end{tabular}
\label{wc1ex}
\end{table}

The remaining columns of the triangle of $w(n,k)$ values all follow the same recurrence, involving just the two terms $w(n-1,k-1)$ and $w(n-2,k)$, a significant improvement to a linear recurrence of degree $2k+4$.

\begin{theorem} \label{wc2}
For $k \ge 2$ and $n \ge 6$, \[w(n,k) = w(n-1,k-1) + w(n-2,k).\]
\end{theorem}

\begin{proof}
We establish a bijection
\[W(n,k) \cong W(n-1,k-1) \cup W(n-2,k)\]
from which the claim follows.

Given a composition in $W(n,k)$, it has either a single part 2 after its last water cell or a run of at least two parts 2 after its last water cell.  If there is a single part 2 after the last water cell, then remove the part 1 beneath the last water cell to give a composition in $W(n-1,k-1)$.  If the composition in $W(n,k)$ has a run of two or more parts 2 after the last water cell, remove the first of those parts 2, leaving a composition in $W(n-2,k)$.

For the reverse map, given a composition in $W(n-1,k-1)$, add another part 1 after the last water cell to make a composition in $W(n,k)$.  Given a composition in $W(n-2,k)$, add a part 2 after the last water cell.
\end{proof}

See Table \ref{wc2ex} for an example of the bijection.

\begin{table}[ht]
\caption{The bijections of the proof of Theorem \ref{wc2} for $n= 9$ and $k = 3$.} 
\medskip
\centering
\begin{tabular}{rcl}
$W(9,3)$ & $\longleftrightarrow$ & $W(8,2) \cup W(7,3)$ \\ \hline
$(2,2,1,1,1,2)$ & & $(2,2,1,1,2)$ \\
$(2,1,2,1,1,2)$ & & $(2,1,2,1,2)$ \\ 
$(2,1,1,2,1,2)$ & & $(2,1,1,2,2)$ \\
$(2,1,1,1,2,1,1)$ & & $(2,1,1,2,1,1)$ \\
$(1,2,1,1,1,2,1)$ & & $(1,2,1,1,2,1)$ \\ 
$(1,1,2,1,1,1,2)$ & & $(1,1,2,1,1,2)$  \\ \cline{3-3}
$(2,1,1,1,2,2)$ & & $(2,1,1,1,2)$
\end{tabular}
\label{wc2ex}
\end{table}

\subsection{Water cell diagonals}
The sequence of diagonal sums
\[d(n) = w(n,0) + w(n-1,1) + w(n-2,2) + \cdots\]
begins $1, 1, 2, 3, 5, 7, 11, 15, 23, 31$ which matches A052955 in \cite{o} with an additional term 1 at the beginning.  Write $D(n)$ for the corresponding set of compositions.

We connect these to compositions where only the first and last parts can be odd, that is, any internal parts are even: Let $C_{ie}(n)$ be the compositions in $C(n)$ with length one or two and the $(c_1, \ldots, c_t)$ with $t \ge 3$ such that parts $c_2, \ldots, c_{t-1}$ are all even.
For example,
\begin{gather*}
C_{ie}(4) = \{(4),(3,1),(2,2),(1,3),(1,2,1)\}, \\ C_{ie}(5) = \{(5),(4,1),(3,2),(2,3),(2,2,1),(1,4),(1,2,2)\}
\end{gather*}
so that $c_{ie}(4) = 5$ and $c_{ie}(5) = 7$.  Note that, for $n$ even, compositions in $C_{ie}(n)$ either have all parts even or both $c_1$ and $c_t$ odd (for compositions with at least two parts).  Similarly, for $n$ odd, exactly one of $c_1$ and $c_t$ is odd when $t \ge 2$.

\begin{theorem} \label{dp}
The water cell array diagonal starting from $w(n,0)$ equals the number of compositions of $n$ with internal parts even, i.e., $d(n) = c_{ie}(n)$. 

The generating function for $d(n)$ is
\begin{equation}
 \sum_{n \ge 0} d(n) q^n = \frac{1-q^2+q^3}{1 - q - 2 q^2 + 2 q^3}. \label{dgf}
 \end{equation}

Further, for $n \ge 1$,
\begin{equation}
d(n) = \begin{cases} 2^m - 1 & \text{if $n = 2m-1$}, \\ 3\cdot 2^{m-1} - 1 & \text{if $n = 2m$}. \end{cases} \label{df}
\end{equation}
\end{theorem}

We defer the proof of \eqref{dgf} to the next section.

\begin{proof}
We establish a bijection
\[D(n) = \bigcup_{k \ge 0} W(n-k,k) \cong C_{ie}(n).\]

Given a composition in $W(n-k,k)$ for some $k \ge 0$, add a part 1 after each water cell to get a composition $c \in W(n,2k)$ where each internal run of 1s has even length.  Then the composition conjugate to $c$ has any internal parts even.

For the reverse map, given a composition in $C_{ie}(n)$, its conjugate has parts at most 2 (since a part 3, for instance, would require an internal part 1) with each internal run of 1s having even length.  Halve each internal run of 1s to produce a composition in $W(n-k,k)$ for some $k \ge 0$.

To establish \eqref{df}, we use $c_{ie}(n)$ rather than $d(n)$ and first consider the case $n = 2m-1$ for some $m \ge 1$.  We establish the bijection
\[C_{ie}(2m-1) \cong C(m) \cup \{C(m) \setminus (m)\} \]
from which the identity follows since $c(m) = 2^{m-1}$.

Given a composition in the first set $C(m)$, double each part and then decrease the first part by one.  In the second set $C(m)$, double each part and then decrease the last part by one (but not $(2m)$ where the first and last part are the only part).  The resulting compositions have only the first part odd, respectively, the last part odd, so they are in $C_{ie}(2m-1)$.

For the reverse map, a composition in $C_{ie}(2m-1)$ has exactly one odd part, the first or the last part.  For those with first part odd, increase that part by one and halve all parts to make one set $C(m)$.  For those with last part odd (excluding the single part composition $(2m-1)$), increase that part by one and halve all parts to make the other set $C(m)$ except for $(m)$.

For the case $n = 2m$, we establish the bijection 
\[C_{ie}(2m) \cong 2C(m) \cup  \{C(m) \setminus (m)\} \]
from which the identity follows.

Given a composition in the first set $C(m)$, double each part.  In the second set $C(m)$, double each part, then decrease the first part by one and add a part 1 at the end; this produces a composition with first part odd and last part 1.  In the third set $C(m)$, double each part, then decrease the first part by one and increase the last part by one (but exclude $(2m)$ which would be sent to itself); this produces a composition with first part odd and last part odd at least 3.

For the reverse map, a composition in $C_{ie}(2m)$ has either no odd parts or two odd parts, the first and the last.  For those with all even parts, halve all parts to make one set $C(m)$.  For compositions with first part odd and last part 1, increase the first part by one, remove the last part, and halve the resulting parts to make a second set $C(m)$.  For compositions with first part odd and last part odd at least 3, increase the first part by one, decrease the last part by one, and halve the resulting parts to make a third set $C(m)$ except for $(m)$.
\end{proof}

See Tables \ref{dbij} and \ref{dfbij} for examples of the bijections.

\begin{table}[ht]
\caption{The first bijection of the proof of Theorem \ref{dp} for $n= 8$ (with only some of the 17 compositions in $W(8,0)$ but all of $W(7,1)$ and $W(6,2)$).} 
\medskip
\centering
\begin{tabular}{rcl}
$W(8,0) \cup W(7,1) \cup W(6,2)$ & $\longleftrightarrow$ & $C_{ie}(8)$ \\ \hline
$(2,2,2,2)$ & & $(1,2,2,2,1)$ \\
$(1,2,2,1,1,1)$ & & $(2,2,4)$ \\ 
$(1,1,2,1,1,1,1)$ & & $(3,5)$ \\
$(1,1,1,1,2,2)$ & & $(5,2,1)$ \\
$(1^8)$ & & $(8)$ \\
$\cdots$ & & [compositions with all internal parts 2] \\ \cline{1-1}
$(2,2,1,2)$ & & $(1,2,4,1)$ \\
$(2,1,2,2)$ & & $(1,4,2,1)$ \\
$(2,1,2,1,1)$ & & $(1,4,3)$ \\
$(1,2,1,2,1)$ & & $(2,4,2)$ \\
$(1,1,2,1,2)$ & & $(3,4,1)$ \\ \cline{1-1}
$(2,1,1,2)$ & & $(1,6,1)$
\end{tabular}
\label{dbij}
\end{table}

\begin{table}[ht]
\caption{The bijections of the proof of \eqref{df} for $n= 5$ and $n = 6$ (both with $m = 3$).} 
\medskip
\centering
\begin{tabular}{rcl}
$2C(3) \setminus (3)$ & $\longleftrightarrow$ & $C_{ie}(5)$ \\ \hline
$(3)$ & & $(5)$ \\
$(2,1)$ & & $(3,2)$ \\ 
$(1,2)$ & & $(1,4)$ \\
$(1,1,1)$ & & $(1,2,2)$ \\ \cline{1-1}
$(2,1)$ & & $(4,1)$ \\
$(1,2)$ & & $(2,3)$ \\
$(1,1,1)$ & & $(2,2,1)$
\end{tabular}
\qquad
\begin{tabular}{rcl}
$3C(3) \setminus (3)$ & $\longleftrightarrow$ & $C_{ie}(6)$ \\ \hline
$(3)$ & & $(6)$ \\
$(2,1)$ & & $(4,2)$ \\ 
$(1,2)$ & & $(2,4)$ \\
$(1,1,1)$ & & $(2,2,2)$ \\ \cline{1-1}
$(3)$ & & $(5,1)$ \\
$(2,1)$ & & $(3,2,1)$ \\ 
$(1,2)$ & & $(1,4,1)$ \\
$(1,1,1)$ & & $(1,2,2,1)$ \\ \cline{1-1}
$(2,1)$ & & $(3,3)$ \\ 
$(1,2)$ & & $(1,5)$ \\
$(1,1,1)$ & & $(1,2,3)$
\end{tabular}
\label{dfbij}
\end{table}

We establish \eqref{dgf} using Riordan arrays in the next section.  One could pursue a combinatorial proof of that generating function, or at least the recurrence
\[d(n) = d(n-1) + 2d(n-2) - 2d(n-3)\]
for $n \ge 3$ that follows from it, similar to the proof of Theorem \ref{wum} (d).  Also, the identity \[d(n) = 2d(n-2)+1\] for $n \ge 2$ follows from \eqref{df}, but there may be simpler combinatorial proof of this identity not requiring cases based on the parity of $n$.  We leave these to the interested reader.

\section{Connections to Riordan arrays}

In this final section, we consider many of our results through Riordan arrays and complete the proof of Theorem \ref{dp}.
For our purposes, it is enough to say that certain triangular arrays of integers can be described by an ordered pair of rational functions $(d(t), h(t))$ where $d(t)$ is the generating function of the first column and $h(t)$ encapsulates the relation between columns.  From this formulation, row sums, diagonal sums, and antidiagonal sums follow directly from $d(t)$ and $h(t)$ as detailed by Sprugnoli \cite{s94}.

The irregular water cell triangle $w(n,k)$ of Table \ref{wtab} is not a Riordan array, but the portion starting with the $k=1$ column from $n=5$ is.  Specifically, the Riordan array for compositions with a positive number of water cells is given by
\begin{equation}
\left(\frac{1}{(1-t)^2(1-t^2)^2}, \frac{1}{1-t^2} \right)\!. \label{w1+}
 \end{equation}
The verification of $d(t)$ is the $k=1$ case Theorem \ref{wcolgf} and confirming $h(t)$ is equivalent to Theorem \ref{wc2}.  With this, the row sums for this subtriangle have generating function
\begin{align}
\frac{d(t)}{1-t h(t)} & = \frac{1}{(1 - t)^2 (1 - t^2) (1 - t - t^2)} \label{wrs} \\
& = 1 + 3 t + 8 t^2 + 17 t^3 + 34 t^4 + 63 t^5 + 113 t^6 + 196 t^7 + 334 t^8 + \cdots \notag 
\end{align}
which matches A344004 in \cite{o}.

Since we know the row sums for the complete $w(n,k)$ array are Fibonacci numbers, we now have another way to establish \eqref{w0gf} from Theorem \ref{wum}: The generating function for $w(n,0)$ is
\[ \frac{1}{1 - t - t^2} - \frac{t^5}{(1 - t)^2 (1 - t^2) (1 - x - t^2)} = \frac{1-t-t^3}{1 - 2 t + 2 t^3 - t^4} \]
where the $t^5$ term modifying \eqref{wrs} places the Riordan array correctly in the $w(n,k)$ triangle.

The diagonal sums of the Riordan array subtriangle of $w(n,k)$ have generating function
\begin{align}
\frac{d(t)}{1-t^2 h(t)} & = \frac{1}{(1 - t)^3 (1 + t) (1 - 2 t^2)} \label{wrd} \\
& = 1 + 2 t + 6 t^2 + 10 t^3 + 21 t^4 + 32 t^5 + 58 t^6 + 84 t^7 +  141 t^8 + \cdots. \notag 
\end{align}

With this, we complete the proof of Theorem \ref{dp}.

\begin{proof}[Proof of \eqref{dgf}]
The generating function for $d(n)$ follows from combining \eqref{w0gf} for $w(n,0)$ and \eqref{wrd} with the appropriate factor for the rest of the diagonal sum.  The first $d(n)$ with two positive summands is $d(6) = w(6,0) + w(5,1)$, 
so we have
\[ \sum_{n\ge0} d(n) t^n = \frac{1-t-t^3}{1 - 2 t + 2 t^3 - t^4} + \frac{t^6}{(1 - t)^3 (1 + t) (1 - 2 t^2)} = \frac{1 - t^2 + t^3}{1 - t -2t^2 + 2t^3}. \qedhere\]
\end{proof}

\section*{Acknowledgments}
We appreciate Paul Barry's assistance with Riordan arrays, James Shapiro's helpful online tool \url{https://riordancalculator.com/}, and the anonymous referees' recommendations.  Let us also mention that Mark Shattuck, working from a preprint of this article, has further developed these ideas \cite{s25}.

\medskip

\noindent MSC2020: 05A17, 11B37

\end{document}